\documentclass[11pt]{amsart}
\usepackage{etex} 

\usepackage{amsmath}
\usepackage{amssymb}
\usepackage{amsthm}
\usepackage{amsfonts}
\usepackage{booktabs}
\usepackage[usenames,dvipsnames,svgnames]{xcolor}
\usepackage{epsfig}
\usepackage{fancyvrb} 
\usepackage{hyperref}
\usepackage{mathrsfs}
\usepackage{mathtools}
\usepackage{pictexwd, dcpic}
\usepackage{stmaryrd}
\usepackage{enumitem} 

\usepackage{tikz}
\usepackage{tikz-cd}
\usepackage{caption}
\usetikzlibrary{decorations.markings}

\newcommand{\RR}{\mathbb{R}}

\newcommand{\ZZ}{\mathbb{Z}}

\newcommand{\vB}{\mathcal{B}}
\newcommand{\vC}{\mathcal{C}}

\newcommand{\vF}{\mathcal{F}}

\newcommand{\vI}{\mathcal{I}}
\newcommand{\vJ}{\mathcal{J}}

\newcommand{\vL}{\mathcal{L}}

\newcommand{\vS}{\mathcal{S}}
\newcommand{\vT}{\mathcal{T}}

\newcommand{\Aut}{\operatorname{Aut}}

\newcommand{\Map}{\operatorname{Map}}
\newcommand{\PMap}{\operatorname{PMap}}

\newcommand{\sm}{\setminus}
\newcommand{\ol}[1]{\overline{#1}}

\newcommand{\eand}{\quad \text{ and } \quad}

\definecolor{lightgrey}{gray}{.85}

\theoremstyle{definition}

\theoremstyle{plain}
\newtheorem{thm}{Theorem}[section]
\newtheorem{main}{Theorem}
\newtheorem{lem}[thm]{Lemma}

\newtheorem{prop}[thm]{Proposition}
\newtheorem{fact}[thm]{Fact}

\newtheorem{conj}[thm]{Conjecture}

\theoremstyle{definition}

\newcommand{\MCT}{\Map^\pm(\vT)}

\begin{document}
\title[The mapping class group of the Cantor tree]{The mapping class group of the Cantor tree has only geometric normal subgroups}
\author{Alan McLeay}
\address{University of Luxembourg}
\email{mcleay.math@gmail.com}
\maketitle

\begin{abstract}
A normal subgroup of the (extended) mapping class group of a surface is said to be geometric if its automorphism group is the mapping class group.  We prove that in the case of the Cantor tree surface, every normal subgroup is geometric.  We note that there is no non-trivial finite-type mapping class group for which this statement is true.  We study a generalisation of the curve graph, proving that its automorphism group is again the mapping class group.  This strategy is adapted from that of Brendle-Margalit and the author for certain normal subgroups in the finite-type setting.
\end{abstract}


\section{Introduction}\label{sec_intro}
Let $\vT$ be the \emph{Cantor tree surface}, that is, the infinite-type surface obtained by taking the boundary of a neighbourhood of an infinite trivalent tree embedded in $\RR^3$, see Figure \ref{inseparable}.  Equivalently, $\vT$ is homeomorphic to removing a Cantor set from a sphere.  Let $\Map^\pm(\Sigma)$ be the \emph{(extended) mapping class group} of a surface $\Sigma$, that is, the group of isotopy classes of homeomorphisms of $\Sigma$.

\begin{main}\label{thm_normal}
If $N$ is a normal subgroup of $\MCT$ then the natural homomorphism
\[
\MCT \to \Aut N
\]
is an isomorphism.
\end{main}

For finite-type surfaces with empty boundary, this statement is either trivial or false \cite{CMM}.  That is, the statement only holds for spheres with up to three punctures.  If we consider finite-type surfaces with boundary, then any Dehn twist about a curve isotopic to the boundary generates a normal subgroup isomorphic to $\ZZ$, and so the statement fails in these cases also.  We may interpret this result as saying that the \emph{extreme homogeneity} of $\vT$ is mirrored by the subgroup structure of its mapping class group.  As an example of this homogeneity, $\vT$ has no subsurfaces of the following type.

\subsection*{Nondisplaceable subsurfaces}  Mann-Rafi define a finite-type subsurface $S \subset \Sigma$ to be a \emph{nondisplaceable} if $f(S) \cap S \neq \emptyset$ for all $f \in \Map^{\pm}(\Sigma)$ \cite{MR}.

This then suggests the following conjecture.

\begin{conj}
The natural map $\Map^{\pm}(\Sigma) \to \Aut N$ is an isomorphism for all $N$ if and only if $\Sigma$ contains no nondisplaceable subsurfaces.
\end{conj}

The surfaces included in this conjecture are infinitely many.  Along with the Cantor tree surface, they contain the well known (and therefore named) Blooming Cantor tree, Loch Ness Monster, Ladder, and Flute surfaces.

The definition of \emph{nondisplaceable subsurfaces} is key in the work of Mann-Rafi, which determines that members of a sizable subclass of \emph{big} mapping class groups have \emph{coarsely bounded} generating sets.  In particular, the existence of a nondisplaceable subsurface implies that the mapping class group itself is not coarsely bounded \cite[Theorem 1.9]{MR}.

\subsection*{Transitivity}
Recall that the curve graph $\vC(\Sigma)$ is the graph whose vertices are isotopy classes of essential simple closed curves; two vertices span an edge if their isotopy classes of curves contain disjoint representatives.  Another example of the homogeneity of $\vT$ is the fact that $\Map^\pm(\vT)$ acts transitively on $\vC(\vT)$.  While not unique to $\vT$, there are only a total of six surfaces where this occurs; the four or five punctured sphere, the zero or once punctured torus, or the zero or once punctured Cantor tree.  In the five cases other than $\vT$ however, the surface contains a nondisplaceable subsurface.

The proof of Theorem \ref{thm_normal} relies on first generalising $\vC(\vT)$ to another graph on which $\MCT$ also acts transitively.

\subsection*{The graph of $n$-holed spheres}
We define the \emph{graph of $n$-holed spheres} $\vS_n(\vT) = \vS_n$ to have vertices corresponding to homotopy classes of compact connected subsurfaces in $\vT$ that are homeomorphic to a sphere with exactly $n \ge 2$ boundary components.  Two vertices span an edge if they correspond to disjoint spheres in $\vT$, see Figure \ref{inseparable}.
\begin{figure}[t]
\centering
\includegraphics[scale=0.8]{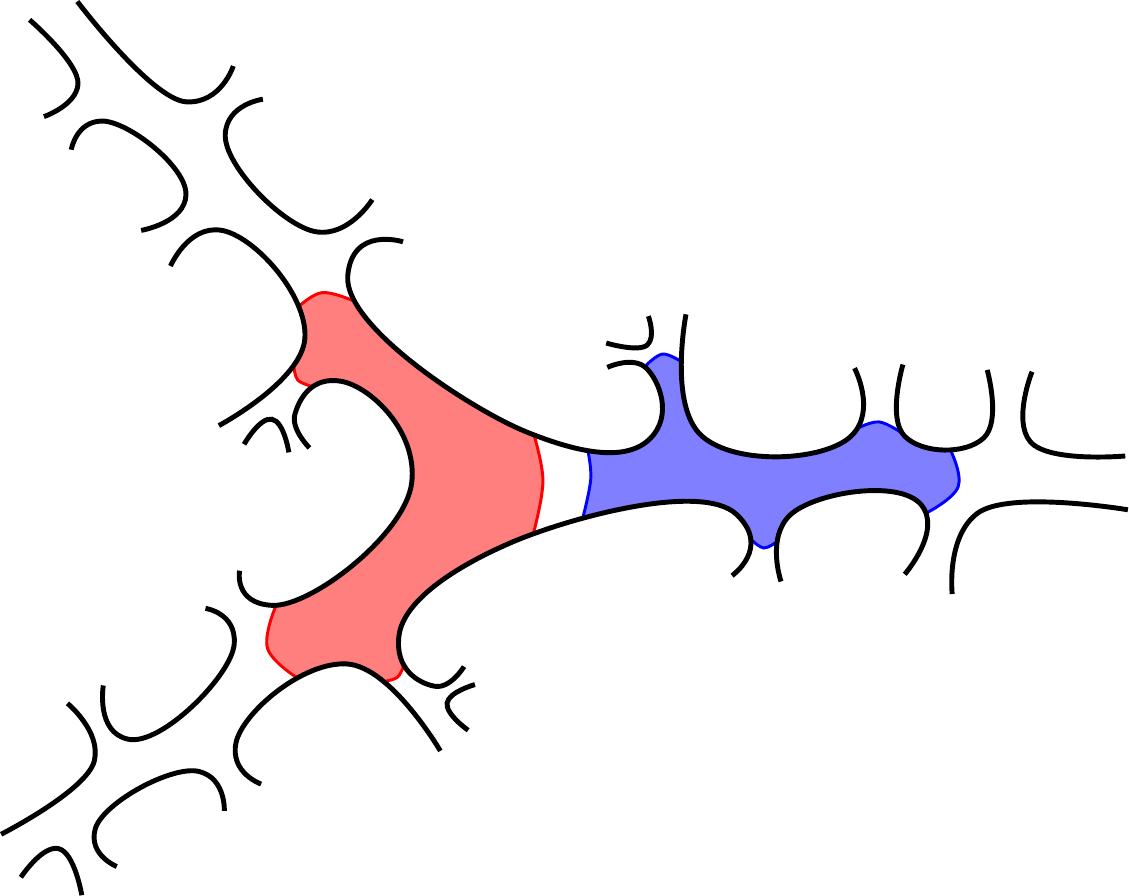}
\caption{Two disjoint $5$-holed spheres in the Cantor tree $\vT$.  The spheres correspond to adjacent vertices in $\vS_n$.  In particular, they satisfy the definition of inseparable spheres used in Lemma \ref{lem_inseparable}.}
\label{inseparable}
\end{figure}

Note that the graph $\vS_2$ is precisely the curve graph $\vC(\vT)$.  The connection between $\vC(\vT)$ and $\vS_n$ is used to prove the following:

\begin{main}\label{thm_graph}
The natural homomorphism
\[
\eta_n : \MCT \to \Aut \vS_n
\]
is an isomorphism for any $n\ge2$.
\end{main}

\subsection{Studied examples}
Before giving some history of these type of results, we give a short list of normal subgroups that have been studied in the literature.  Some definitions are specific to $\MCT$, while others apply to all big mapping class groups.

\subsection*{The (pure) mapping class group}
The group $\Map(\Sigma)$ is the subgroup of $\Map^\pm(\Sigma)$ consisting of only orientation-preserving homeomorphisms.  The group $\PMap(\Sigma)$ is the subgroup of $\Map(\Sigma)$ consisting of the homeomorphisms that fix every end of $\Sigma$.  We note that the automorphism groups of both $\Map(\Sigma)$ and $\PMap(\Sigma)$ have already been proved to be $\Map^\pm(\Sigma)$ in greater generality \cite{PV18}, \cite{HMV19}, \cite{BDR18}.


\subsection*{The compactly supported mapping class group}
The group $\Map_c(\Sigma)$ is the subgroup of $\PMap(\Sigma)$ consisting of only the elements with compact support.  One may equip $\PMap(\Sigma)$ with the compact-open topology, in doing so it has been shown by Patel-Vlamis that $\PMap(\Sigma)$ is precisely the closure of $\Map_c(\Sigma)$ when $\Sigma$ has at most one end \emph{accumulated by genus} \cite[Theorem 4]{PV18}.  In the case of Cantor tree surface therefore, $\PMap(\vT) = \ol {\Map_c(\vT)}$.

\subsection*{The Torelli group}
The \emph{Torelli group} $\vI(\Sigma)$ is the kernel of the natural map $\Map(\Sigma) \to \Aut H_1(\Sigma,\ZZ)$.  As with $\Map(\Sigma)$ the automorphism group of $\vI(\Sigma)$ has already been proven to be $\Map^\pm(\Sigma)$ for all infinite-type surfaces by Aramayona-Ghaswala-Kent-McLeay-Tao-Winarski \cite{AGKMTW19}.  In the same paper, it is shown to be topologically generated by compactly supported elements.

\subsection*{The Johnson kernels}
The \emph{Johnson kernels} $\vJ_k(\Sigma)$ are the kernels of the homomorphism $\Map(\Sigma) \to \Aut \Gamma_1 / \Gamma_k$, where $\Gamma_k$ is a term in the lower central series of $\Gamma_1 = \pi_1(\Sigma)$.  Following the argument of Farb \cite{Farb06}[Theorem 5.10] we see that $\vJ_k(\Sigma)$ contains elements that are supported in four-holed spheres.  It is likely therefore that one could use the graph $\vS_4(\Sigma)$ to extend this result to all surfaces of infinite-type.

\subsection*{A non-example}
Funar-Kapoudjian define a \emph{rigid structure} on $\vT$; a pants decomposition of $\vT$ with specific \emph{seems} between each component \cite{FK04}.  A mapping class $f$ is \emph{asymptotic} if there exists a compact subsurface $S$ such that $f$ preserves the rigid structure in the complement of $S$.  The \emph{asymptotic} (or universal) mapping class group $\vB$ is the corresponding subgroup.  Two reasons $\vB$ has received attention are its strong connection to Thompson's group, and that it contains every genus zero mapping class group as a subgroup.  Furthermore, Funar-Kapoudjian show that $\vB$ is finitely presented \cite{FK04}[Theorem 1].

The subgroup $\vB$ however, is not normal in $\MCT$.  Indeed, suppose $g$ is asymptotic and swaps the infinite sets $P_1, P_2$ of curves contained in the rigid structure.  One may then take $f$ to be the product of twists about another set of curves which intersects $P_1$ infinitely many times and $P_2$ zero times.  Then the element $f^{-1}gf$ fails to fix the curves in $P_1,P_2$ and so is not asymptotic.

\subsection{Ivanov's metaconjecture}
The strategy of this paper draws inspiration from the famous result of Ivanov which, together with work of Korkmaz and Luo, states that if $\Sigma$ is of finite-type (other than finitely many sporadic cases) then $\Aut \vC(\Sigma) \cong \Map^{\pm}(\Sigma)$ \cite{Ivanov97}.  An algebraic characterisation of Dehn twists is then used to show that any automorphism of $\Map^\pm(\Sigma)$ is induced by an automorphism of $\vC(\Sigma)$, and that $\Aut  \Map(\Sigma) \cong \Map^{\pm}(\Sigma)$.

Farb-Ivanov and Brendle-Margalit then realised that by choosing complexes suited to other normal subgroups, they could adapt this strategy.  In doing so, they proved that the Torelli group and the first Johnson kernel are both geometric \cite{FI05}, \cite{BM04}.  This led to \emph{complexes of regions} which were used, first by Brendle-Margalit and then the author, to prove that many normal subgroups in the finite-type setting are geometric \cite{BM17}, \cite{McLeay17}, \cite{McLeay19}.

Prior to the definition of complexes of regions, a series of results
(for example \cite{Bowditch16}, \cite{BM04}, \cite{Irmak06}, \cite{Kida11}, \cite{KP10}, \cite{KP12}, \cite{Margalit04}, and \cite{MP13})
led to Ivanov make the following metaconjecture
\begin{conj}[Ivanov \cite{Farb06}]
Every object naturally associated to a surface $\Sigma$ and having sufficiently rich structure has $\Map^\pm(\Sigma)$ as its group of automorphisms.  Moreover, this can be proved by a reduction to the theorem about automorphisms of $\vC(\Sigma)$.
\end{conj}
The results of Brendle-Margalit and the author therefore resolve the metaconjecture for an infinite family of complexes.  We also note that recent work of Agrawal-Aougab-Chandran-Loving-Oakley-Shapiro-Xiao resolves the metaconjecture for a different infinite family of complexes; namely curve graphs where adjacency is defined by pairs of curves with intersection of at least $k$ \cite{AACLOSX}.  A rephrasing and generalisation of a conjecture of Brendle-Margalit \cite{BM17}[Conjecture 1.5] is as follows
\begin{conj}\label{conj_connected}
Let $\vC_R(\Sigma)$ be a complex of regions, and let $N$ be a normal subgroup of $\Map^\pm(\Sigma)$.
\begin{enumerate}
\item The natural map $\Map^{\pm}(\Sigma) \to \Aut \vC_R(\Sigma)$ is an isomorphism if and only if $\vC_R(\Sigma)$ is connected and admits no exchange automorphisms.
\item The natural map $\Map^{\pm}(\Sigma) \to \Aut N$ is an isomorphism if and only if $N$ contains elements whose supports $L,R \subset \Sigma$ are disjoint and separated by a pair of pants.
\end{enumerate}
\end{conj}

\subsection*{The leap to big mapping class groups}
Of course, there is no reason that the metaconjecture must be restricted to surfaces of finite-type.  Indeed, it has been shown by Hernandez-Morales-Valdez \cite{HMV19} and Bavard-Dowdall-Rafi \cite{BDR18} that the original result for the curve graph also holds for any surface of infinite-type.  Similarly, the result for \emph{big} Torelli groups noted above relies on the equivalent result for the \emph{Torelli complex}, also adapted from the finite-type case \cite{AGKMTW19}, \cite{BM04}.


\subsection{Outline of the paper}
We begin by proving Theorem \ref{thm_graph} in Section \ref{section_graph}.  In Section \ref{section_elements} we study the elements of a normal subgroup and give an algebraic characterisation of elements supported on $n$-holed spheres.  Finally in Section \ref{section_automorphisms} we complete the proof of Theorem \ref{thm_normal}.

We note here that this algebraic characterisation can be adapted for any finite-index subgroup of $N$, and therefore the proof of Theorem \ref{thm_normal} can be modified for the abstract commensurator group of $N$.  This is omitted from the paper to simplify some of the definitions in Section \ref{section_elements}.

\subsection*{Acknowledgements}
The author would like to thank Javier Aramayona, Tara Brendle, Ty Ghaswala, Autumn Kent, Dan Margalit, Jing Tao, and Becca Winarski for helpful discussions on the techniques used in this paper.
He is grateful to Nick Vlamis for comments on an earlier draft.
He was supported by the COALAS FNR AFR bilateral grant.

\section{The graphs}\label{section_graph}

In this section we will prove Theorem \ref{thm_graph} for $n\ge3$ by studying its relationship to the curve graph.  To that end, we note that if $\vL_R$ is the link of a vertex $R \in \vS_n$ then there is a notion of \emph{sides} of $\vL_R$, first introduced by Ivanov \cite{Ivanov97}.

For any $L \in \vS_n$ disjoint from $R$ we can define the set $\vL_R(L) \subset \vL_R$ to be all spheres that lie in the same component of $\vT \sm R$ as $L$.  One can see that $S \in \vL_R(L)$ if and only if there exists a vertex in $\vL_R$ that is nonadjacent to both $S$ and $L$.

\subsection*{Inseparable spheres}  We say that two disjoint spheres $L,R \in \vS_n$ are \emph{inseparable} if there exists a unique curve $c$ belonging to the boundaries of both $L$ and $R$, see Figure \ref{inseparable}.

\begin{lem}\label{lem_inseparable}
Two vertices $L,R \in \vS_n$ are inseparable spheres if and only if they span an edge and $\vL_L(R) \cap \vL_R(L)$ is empty.
\end{lem}

\begin{proof}
Since $L$ and $R$ are disjoint, there exists a component $X$ of $\vT \sm \{L,R\}$ such that $\partial X = \{c_L, c_R\}$ where $c_L \in \partial L$ and $c_R \in \partial R$.  By definition, the vertices of $\vL_L(R) \cap \vL_R(L)$ are exactly the vertices contained in $X$.  It is clear that $X$ contains no vertices of $\vS_n$ if and only if $c_L = c_R$.
\end{proof}

The following result implies that the map from Theorem \ref{thm_graph} is injective.  Moreover, it will be used to show injectivity for map from Theorem \ref{thm_normal}.

\begin{lem}\label{lem_inj}
If $f \in \MCT$ fixes every $n$-holed sphere then $f$ is the identity.
\end{lem}

\begin{proof}
If $c_1,\dots, c_n$ are the boundary components of an $n$-holed sphere then $f(c_i)=c_j$.  For any curve $c$ we can find inseparable spheres so that $c$ is the unique curve in the boundary of both.  It follows that $f(c)=c$.  Since $c$ is arbitrary, we have that $f$ fixes every curve and so by the Alexander method for infinite-type surfaces \cite{HMV19}, $f$ is the identity.
\end{proof}

We now prove the map $\MCT \to \Aut \vS_n$ is an isomorphism.

\begin{proof}[Proof of Theorem \ref{thm_graph}]
We assume that $n\ge 3$.  By Lemma \ref{lem_inj} the map is injective.  We now show that there is an injective map $\Aut \vS_n \to \Aut \vC(\vT)$.

Lemma \ref{lem_inseparable} shows that every automorphism $\phi$ of $\vS_n$ induces a permutation of the vertices of $\vC(\vT)$.  This permutation can be extended to an automorphism $\phi_* \in \Aut \vC(\vT)$, as two curves are disjoint if and only if they are boundary components of some $n$-holed sphere.  The fact that this map is a well defined homomorphism is left as an exercise for the reader.

If $\phi_*$ is the identity, then $\phi$ fixes every pair of inseparable spheres.  It follows that $\phi$ is the identity and so the map $\Aut \vS_n \to \Aut \vC(\vT)$ is injective.  We now have that
\[
\MCT \hookrightarrow \Aut \vS_n \hookrightarrow \Aut \vC(\vT) \cong \MCT
\]
is the identity, and so all the maps are isomorphisms.
\end{proof}

\section{Elements of normal subgroups}\label{section_elements}
In order to prove Theorem \ref{thm_normal} we will give a characterisation of elements of a normal subgroup $N$ whose supports are $n$-holed spheres.  The first step is to characterise a subset of $N$ containing only elements of finite support.  Of course, it may be that $N$ only contains such elements, in which case $N$ is countable.  Conversely, a result of Bavard-Dowdall-Rafi implies that if $N$ contains an element of infinite support, then it contains uncountably many such elements \cite[Proposition 4.2]{BDR18}.  We will adapt their proof to give the characterisation we require.

\begin{prop}\label{prop_finite}
If $f \in N$ has a countable conjugacy class then $f^2$ has finite support.
\end{prop}

\begin{proof}
First, suppose that $f^2$ fixes every curve in the complement of a compact subsurface $S$.  By the Alexander method for infinite-type surfaces, we have that the support of $f^2$ is contained in $S$.  It suffices then to show that if $f$ has infinite support, and $f^2$ does not satisfy the above condition, then the conjugacy class in $N$ of $f$ is uncountable.

To that end, let $\vC = (c_1, c_2, \dots)$ be a set of curves such that $f(c_i)$ and $c_j$ are disjoint, except possibly when $i=j$, \cite[Proposition 4.2]{BDR18}.  Furthermore, we may choose each $c_i$ such that $f(c_i)$ and $f^{-1}(c_i)$ are distinct, as otherwise $f^2$ would fix every curve in the complement of a compact subsurface.  For each $\epsilon = (\epsilon_1, \epsilon_2, \dots )$, where $\epsilon_i \in \{0,1\}$ let $T_\vC^\epsilon = \prod T_{c_i}^{\epsilon_i}$.  Now let
\[
h_\epsilon = f^{-1} T_\vC^{-\epsilon} f T_\vC^\epsilon \eand f_\epsilon = h_\epsilon^{-1} f h_\epsilon.
\]

Note that $h_\epsilon$ is in $N$ and so $f_\epsilon$ is a conjugate in $N$ of $f$.  Now, if $f(\vC) = (f(c_1), f(c_2), \dots)$ then $h_\epsilon = T_{f^{-1}(\vC)}^{-\epsilon}T_\vC^{\epsilon} = \prod T_{f^{-1}(c_i)}^{-\epsilon_i}T_{c_i}^{\epsilon_i}$.

For distinct $\epsilon, \epsilon'$ we now have
\begin{align*}
h_{\epsilon'} f_\epsilon f_{\epsilon'}^{-1} h_{\epsilon'}^{-1} &= h_{\epsilon'} h_{\epsilon}^{-1} ( f h_\epsilon f^{-1} ) ( f h_{\epsilon'}^{-1} f^{-1}) \\ 
&= (T_{f^{-1}(\vC)}^{-\epsilon'} T_\vC^{\epsilon'})  (T_{\vC}^{-\epsilon} T_{f^{-1}(\vC)}^{\epsilon})  (T_{\vC}^{-\epsilon} T_{f(\vC)}^{\epsilon})   (T_{f(\vC)}^{-\epsilon'} T_{(\vC)}^{\epsilon'}) \\  
&= \prod (T_{f^{-1}(c_i)}^{-1}T_{c_i})^{(\epsilon'_i - \epsilon_i)} (T_{c_i}^{-1}T_{f(c_i)})^{(\epsilon_i -\epsilon'_i)}.
\end{align*}

Since $f(c_i) \neq f^{-1}(c_i)$ for all $i$ we have that the product is not the identity, and therefore $f_\epsilon \neq f_{\epsilon'}$.
\end{proof}

We note that there is likely an argument to say that every element of infinite support has an uncountable conjugacy class; although our characterisation does not depend on this.  Implicit in the above proof is the fact that every normal subgroup contains elements with finite support.  Following from Proposition \ref{prop_finite} we focus our attention solely on such elements.

To that end, we denote by $\vF_N$ the squares of all elements of $N$ with countable conjugacy classes.  In particular, we will make use of elements which have been called \emph{pure} in the literature \cite{Ivanov92}.  This is not equivalent to an element that fixes the ends of the surface, which is also called \emph{pure}.

\subsection*{Pure (in the sense of Ivanov)}
We refer to an element $f$ with finite support as \emph{pure (in the sense of Ivanov)}, or PSI, if 
\begin{enumerate}
\item $f$ can be written as the product $f_1f_2\dots f_s$,
\item each $f_i$ is a power of a Dehn twist or a partial pseudo-Anosov, and
\item $f_i$ and $f_j$ have disjoint support for all $i$ and $j$.
\end{enumerate}
The mapping classes $f_1, f_2, \dots f_s$ are called the \emph{components} of $f$.  Throughout this section we will use the three following results, often without mention.  For more detail, see \cite{Ivanov92}.

\begin{fact}\label{fact_powers}
If $f$ has finite support then there exists some $k>0$ such that $f^k$ is PSI.
\end{fact}

\begin{fact}\label{fact_commute}
Let $f$ and $g$ be two commuting PSI elements.  Then $f_i$ and $g_j$ have disjoint support, or $f_i$ and $g_j$ are powers of the same partial pseudo-Anosov for all $i,j$.
\end{fact}
If $f \in \vF_N$ then we write $S_f$ for its support and $f^g$ for the conjugate $g^{-1} f g$.
\begin{fact}
If $f$ is PSI and $g$ is a mapping class, then $S_{f^g} = g(S_f)$.
\end{fact}

The remainder of this section gives a characterisation of elements whose supports are $n$-holed spheres.  Many of the ideas are adapted from Brendle-Margalit \cite{BM17}.  A key difference however is that in the finite-type setting, there exists a finite-index normal subgroup of any mapping class group consisting solely of PSI elements.

\subsection*{Vertex pairs}
We call two elements $a,b \in \vF_N$ an \emph{overlapping pair} if there exist no powers of $a$ and $b$ that commute.  For any overlapping pair $(a,b)$ we define the set of \emph{commuting overlapping pairs}
\[
COP_N(a,b) = \{ (c,d) \mbox{ overlapping pair }\ |\ c,d \in C_N(a,b) \}.
\]
Finally, we say that an overlapping pair $(a,b)$ is a \emph{vertex pair} if $COP_N(a,b)$ is maximal with respect to inclusion among all overlapping pairs.  We leave the proof of the following lemma as an exercise for the reader.

\begin{lem}\label{lem_fc}
If $(a,b)$ is  a vertex pair in $N$, then for any $k>0$ the sets $COP_N(a^k,b^k)$ and $COP_N(a,b)$ are equal.  Hence $(a^k, b^k)$ is also vertex pair in $N$.
\end{lem}

For any vertex pair $(a,b)$, we define the subsurface $X_{(a,b)}$ to be the support of $COP_N(a,b)$.  Finally we define
\[
V_{(a,b)} = \vT \sm X_{(a,b)}.
\]

\begin{lem}\label{lem_spheres}
If $(a,b)$ is a vertex pair in $N$ then $V_{(a,b)}$ is homeomorphic to an $n$-holed sphere.  Moreover, if $(c,d)$ is another vertex pair then $V_{(c,d)}$ is homeomorphic to $V_{(a,b)}$.
\end{lem}

\begin{proof}
Let $a^k, b^k$ be PSI and let $a_1,\dots,a_s$ and $b_1,\dots,b_t$ be the respective components.  There exists some $a_j$ and some $I \subset \{1, \dots, t\}$ such that $a_j$ fails to commute with $b_i$ if and only if $i \in I$.  We define $\hat b = \prod_{i \in I}b_i$.  Moreover, we may choose $k$ such that $a_j$ and $\hat b$ generate a free group.  Now, we define $\bar a = [a_j, b]$ and $\bar b = [a_j, b^{-1}]$.  Note that by construction $\bar a, \bar b \in \vF_N$ and can be written in terms of $a_j$ and $\hat b$. Furthermore, both $\bar a$ and $\bar b$ are partial pseudo-Anosov and no powers of $\bar a$ and $\bar b$ commute, so they are an overlapping pair.

If $(c,d) \in COP_N(a,b)$ then $c,d$ commute with both $a_j$ and $\hat b$, hence $\bar a$ and $\bar b$.  It follows that $COP_N(a,b) \subset COP_N(\bar a, \bar b)$.  Since $COP_N(a,b)$ is maximal we have that $COP_N(a,b) = COP_N(\bar a, \bar b)$.  We will write $V$ and $W$ for $V_{(a,b)}$ and $V_{(\bar a, \bar b)}$ respectively, and so we have shown that $V = W$.

Suppose $W$ is of infinite-type.  Since $\bar a$ and $\bar b$ are of finite-type, there exists an overlapping pair in $W$ whose supports are both disjoint from $\bar a$ and $\bar b$ this is a contradiction, so $W$ must be of finite-type.  Since $\bar a$ and $\bar b$ are partial pseudo-Anosov they each have connected support contained in $W$.  If $W$ has two or more disjoint components then $\bar a$ and $\bar b$ have disjoint support, hence they commute.  Again, this is a contradiction, so $W$ is homeomorphic to an $n$-holed sphere.

Finally, let $(c,d)$ be a vertex pair such that $V_{(c,d)}$ is homeomorphic to an $m$-holed sphere.  Up to relabeling it follows that there exists a mapping class $f$ such that $f(V_{(c,d)}) \subset V_{(a,b)}$.  We now have that $(c^f, d^f)$ is a vertex pair and $COP_N(a, b) \subset COP_N(c^f,d^f)$.  Since $COP_N(a,b)$ is maximal the two sets are equal, so $V_{(a,b)} = V_{(c^f,d^f)} = f(V_{(c,d)})$, hence $m=n$.
\end{proof}

Lemma \ref{lem_spheres} tells us that vertex pairs define reasonable candidates for the vertices for some $\vS_n$. The next lemma completes the picture, by showing that the commutativity of two vertex pairs is determined by their supports.

\begin{lem}\label{lem_adj}
Given two vertex pairs $(a,b)$ and $(c,d)$, the spheres $V_{(a,b)}$ and $V_{(c,d)}$ are disjoint if and only if $(c,d) \in COP_N(a,b)$.
\end{lem}

\begin{proof}
The forward direction follows from Fact \ref{fact_commute} and the definition of vertex pair.  To prove the other direction we note that from Lemma \ref{lem_fc} we can assume that $a,b,c,d$ are all PSI.  Since no powers of $c$ and $d$ commute, Fact \ref{fact_commute} says that, up to relabeling, the supports of $a$ and $c$ are disjoint and the supports of $b$ and $d$ are disjoint.

If the supports of $a$ and $d$ intersect, then from Fact \ref{fact_commute} they are powers of the same partial pseudo-Anosov.  There therefore exists some $k>0$ such that $a^k=d^k$.  Since no powers of $a$ and $b$ commute, we have that $b$ and $d^k$ do not commute.  This is a contradiction, so the supports of $a$ and $d$ do not intersect, nor do the supports of $b$ and $c$.  It follows that $V_{(a,b)}$ and $V_{(c,d)}$ are disjoint.
\end{proof}

\section{Automorphisms of normal subgroups}\label{section_automorphisms}
Following from the results of the previous section, we can now define the map
\[
\Psi : \Aut N \to \Aut \vS_n
\]
that sends $\alpha \in \Aut N$ to $\alpha_* \in \Aut \vS_n$, where
$\alpha_* (V_{(a,b)}) = V_{(\alpha(a),\alpha(b))}$.

\begin{lem}\label{lem_welldefinedinjective}
The map $\Psi$ defined above can be extended to an injective homomorphism.
\end{lem}

\begin{proof}
We first verify that the map makes sense.  First, $V_{(a,b)}$ is an $n$-holed sphere from Lemma \ref{lem_spheres}.  By definition $(\alpha(a),\alpha(b))$ is a vertex pair and so again $V_{(\alpha(a),\alpha(b))}$ is an $n$-holed sphere.  We now show that $\Psi$ is well defined.

We must show that if $V_{(a,b)} = V_{(c,d)}$ then $\alpha_* ( V_{(a,b)} ) = \alpha_* ( V_{(c,d)} )$.  This is equivalent to showing that
\[
COP_N(\alpha (a), \alpha (b)) \subset COP_N(\alpha (c), \alpha (d))
\]
Suppose $(f,g) \in COP_N(\alpha (a), \alpha (b))$.  Since $\alpha$ is an isomorphism, $\alpha^{-1}(f),\alpha^{-1}(g)$ commute with $a$ and $b$.  By Lemma \ref{lem_adj} we have that
\[
(\alpha^{-1}(f), \alpha^{-1}(g)) \in COP_N(c,d).
\]
Again, since $\alpha$ is an isomorphism we conclude that $f,g$ commute with $\alpha (c),\alpha(d)$.  It follows that $\Phi$ is a well defined set map.  By Lemma \ref{lem_adj} however, this extends to a well defined automorphism.  It remains to show that $\Psi$ is injective.

We assume that $\alpha_*$ is the identity and we will show that $\alpha(f) = f$ for all $f \in N$.  As before, we denote by $\alpha(f)_*$ and $f_*$ the images of $\alpha(f)$ and $f$ in $\vS_n$ under the injective map $\eta_n : \MCT \hookrightarrow \Aut \vS_n$.  We will show that $\alpha(f)_* = f_*$.  To that end, let $(a,b)$ be a vertex pair.
\begin{align*}
f_*(V_{(a,b)}) &= V_{(a^f, b^f)} = \alpha_* ( V_{(a^f, b^f)} ) = V_{ ( \alpha (a^f), \alpha ( b^f ) )} \\ 
&= V_{ ( \alpha (a)^{\alpha(f)}, \alpha(b)^{\alpha(f)} ) } = \alpha(f)_*(V_{(\alpha(a),\alpha(b))}) \\ 
&= \alpha(f)_* \alpha_* (V_{(a,b)}) = \alpha(f)_* (V_{(a,b)}).
\end{align*}
Since $\alpha(f)_* = f_*$ and $\eta_n$ is injective, we have shown that $\alpha(f) = f$, completing the proof.
\end{proof}

We now combine Theorem \ref{thm_graph} and Lemma \ref{lem_welldefinedinjective} to complete the proof of the main result.

\begin{proof}[Proof of Theorem \ref{thm_normal}]
We first show that the natural map $\MCT \hookrightarrow \Aut N$ is injective.  Suppose $f \in \MCT$ is such that $a^f = a$ for all $a \in N$.  Then from Lemma \ref{lem_spheres} $f$ fixes every $n$-holed sphere.  Injectivity therefore is implied by Lemma \ref{lem_inj}.

Since the composition of injective maps
\[
\MCT \hookrightarrow \Aut N \hookrightarrow \Aut \vS_n \cong \MCT
\]
is the identity homomorphism, the result follows.
\end{proof}

\bibliography{alanbib}{}

\bibliographystyle{plain}
\end{document}